\documentclass[11pt]{article}
\usepackage{amsmath, amssymb, amsthm}
\usepackage[bookmarksnumbered, colorlinks, , linkcolor=black, citecolor=blue, urlcolor=blue, plainpages]{hyperref}
\usepackage{graphicx}
\usepackage{hyperref}
\usepackage{url}
\newtheorem{teo}{Theorem}
\newtheorem{prop}[teo]{Proposition}
\newtheorem{lema}[teo]{Lemma}
\newtheorem{cor}[teo]{Corollary}
\theoremstyle{definition} %coloca o título em negrito e o corpo normal

\newcommand{\RR}{(\mathbb{R}^{n})}

\DeclareMathOperator{\Conv}{Conv}

\DeclareMathOperator{\I}{I}

\DeclareMathOperator{\dom}{dom}
\DeclareMathOperator{\interior}{int}

\DeclareMathOperator{\MA}{MA}

\DeclareMathOperator{\D}{D^2}

\DeclareMathOperator{\LC}{Conv_{ld}}

\DeclareMathOperator{\bd}{bd}
\DeclareMathOperator{\lip}{Conv_{lip}}
\newcommand{\finite}{(\mathbb{R}^n; \mathbb{R})}

\DeclareMathOperator{\supp}{supp}
\DeclareMathOperator{\conc}{Conc([0,\infty))}

\newcommand{\dif}{\mathop{}\!\mathrm{d}}

\title{Asymptotic Weighted Approximation of  Convex Functions}

\author{Fernanda Moreira Baêta%
  \thanks{Institut für Diskrete Mathematik und Geometrie, Technische Universität Wien, Wiedner Hauptstraße 8-10/1046, 1040 Wien, Austria. 
  Email address: \texttt{fernanda.baeta@tuwien.ac.at}}
}
\date{} 

\begin{document}
\maketitle
\pretolerance10000
\begin{abstract}
Extending classical results on polytopal approximation of  convex bodies, we derive asymptotic formulas for the weighted approximation  of smooth convex  functions by  piecewise affine convex functions as the number of their facets
tends to infinity. These asymptotic expressions are formulated in terms of a functional that extends the notion of affine surface area to the functional setting.
\end{abstract} \maketitle

\section{Introduction}
The \textit{affine surface area}  of a convex body (i.e., a non-empty compact convex set) $C\subset\mathbb{R}^{n+1}$ is defined by
\begin{align*}
    \Omega(C)= \int_{\bd C}\kappa_{C}(x)^\frac{1}{n+2}\dif \mathcal{H}^{n}(x),
\end{align*}
where   $\kappa_C$ denotes the  (generalized) Gaussian curvature of the convex body $C$, $\bd C$ is its topological boundary,  and $\mathcal{H}^{n}$ is the $n$-dimensional Hausdorff measure. In the theory of polytopal approximation, the affine surface area provides valuable insight into the quality of approximation of a convex body by polytopes.   Let $\delta(\cdot, \cdot)$ denote the \textit{symmetric difference metric} on the set of convex bodies $\mathcal{K}^{n+1}$,  i.e.,  $\delta(C, D)$ for $C,D\in \mathcal{K}^{n+1}$ is the $(n+1)$-dimensional volume of their symmetric difference $C\cup D\backslash C\cap D$. For a convex body $C$, let
$\mathcal{P}_{(m)}^c= \mathcal{P}_{(m)}^c(C)$ denote the  class of all convex polytopes in $\mathcal{K}^{n+1}$ that are circumscribed to the convex body $C$  and have at most $m$ facets. In \cite{gruber1993asymptotic}, Gruber showed that 
\begin{align}\label{gruber}
    \delta(C, \mathcal{P}_{(m)}^c)  \sim \frac{\delta_{ n}}{2}\Omega(C)^\frac{n+2}{n}\frac{1}{m^\frac{2}{n}}
\end{align}
as $m\rightarrow +\infty$, provided that  $C$ has a boundary of differentiability class $C^2_+$.  Here, the  constant $\delta_{n}$ depends only on the dimension $n+1$. However, its explicit value is not known for  $n>3$. Gruber also obtained an analogous result for the class of all convex polytopes with at most $m$ vertices in $\mathcal{K}^{n+1}$ that are inscribed in $C$ (see \cite{gruber1993asymptotic}).

The assumption   $\kappa_C>0$  was later removed by  Böröczky \cite{boroczky2000approximation}. Furthermore, Ludwig \cite{ludwig1999affine} showed that, in the case $n=2$, the requirement that $C$ be of class $C^2$ can also be omitted. In addition, Ludwig \cite{ludwig1999asymptotic} established general approximation results for the symmetric difference metric without restricting the approximating polytopes to be either inscribed  in or circumscribed around $C$, and also considered weighted cases.  For more information, see \cite[Chapter 11]{gruber2007convex}.

%Further related results can be found in \cite{gruber1982approximation, mcclure1975polygonal, fejes1972lagerungen}.

Recently, notions from convex geometry have been adapted to functions, opening new perspectives on convex function spaces.
For a convex function $u:\mathbb{R}^n\to (-\infty, +\infty]$, the (effective) \textit{domain} of $u$ is defined by $\dom u=\{x\in\mathbb{R}^n: u(x)<+\infty\}$. Throughout, $\|x\|$ denotes the Euclidean norm of $x \in \mathbb{R}^n$, and $\interior(A)$ denotes the interior of a set $A\subset \mathbb{R}^n$.

Let
\begin{multline}\label{def_set}
 \LC\RR=\{u:\mathbb{R}^n\to (-\infty, +\infty]\mid \ \dom u \text{ is a compact set, } \\
 u \text{ is Lipschitz on  } \interior (\dom u)\},   
\end{multline}
denote the set of convex functions $u: \mathbb{R}^n\to (-\infty,+\infty]$ with compact domain that are Lipschitz on the interior of their domains.  Note that if $\dom u$ is not full-dimensional, $u$ is not required to be Lipschitz. The space $\LC(\mathbb{R}^n)$ is endowed with the following notion of convergence.  
For $u, u_k \in \LC(\mathbb{R}^n)$, $k \in \mathbb{N}$, we say that  $u_k$ is $\tau$-convergent to $u$ if
\begin{itemize}
\item[(i)] $u_k$ epi-converges to $u$;
\item[(ii)] the Lipschitz contants $L_{u_k}$ of $u_k$ are uniformly bounded by some $M > 0$ independent of $k$.
\end{itemize}
See Section~\ref{tools} and \cite{zador1982asymptotic} for details on epi-convergence.
\medskip

Consider the following class of functions
\begin{align*}
\conc=
\left\{\zeta: [0,\infty)\rightarrow [0,\infty) \mid \ \zeta \ \mbox{is concave}, \ \lim_{t\rightarrow 0}\zeta(t)=0,   \lim_{t\rightarrow \infty}\zeta(t)/t =0\right\}.    
\end{align*}
For $u \in \LC\RR$ and $\zeta \in \conc$, we define the functional
\begin{align}\label{affine}
    Z_{\zeta}(u)= \int_{\dom u} \zeta(\det \D u(x))\dif x,
\end{align}
where $\det \D u(x)$ is the Hessian matrix of $u$ at $x$.
By a classical result of Aleksandrov \cite{aleksandrov1939second}, convex functions are twice differentiable almost everywhere, and hence their Hessians exist almost everywhere in the domain. For the special case $\zeta(t) = t^{\frac{1}{n+2}}, t\geq 0$, the functional \eqref{affine} coincides with the affine surface area of the graph of the convex function $u \in \LC\RR$, with the convention that $\det \D u(x) = 0$ at points where $u$ is not twice differentiable (see \cite{trudinger2000bernstein}). Moreover, by \cite[Theorem~1]{baeta_ludwig_semicontinuity}, the functional \eqref{affine} is finite and $\tau$-upper semicontinuous, that is, if $u_k\in \LC\RR$ is $\tau$-convergent to $u\in \LC\RR$, then
\begin{align*}
 Z_{\zeta}(u)\geq \limsup_{k\to +\infty} Z_{\zeta}(u_k). 
\end{align*}

We say that $\mathcal{T}$ is a triangulation of a bounded set $B\subset \mathbb{R}^n$ if it decomposes $B$ into $m$-simplices contained in $B$ such that the intersection of any two simplices in $\mathcal{T}$ is either a common lower-dimensional simplex or empty.    Note that such a triangulation determines a set of vertices inside $B$, and the union of the simplices is not required to be equal $B$.   Let $P_m = P_m(v)$ denote the class of piecewise affine interpolations of $v: B\to\mathbb{R}$ on a triangulation $\mathcal{T}\in \mathcal{T}_m$   on a bounded convex set $B\subset \mathbb{R}^n$, where $\mathcal{T}_m$ is the set of all triangulations with at most $m$ vertices, and  the diameter of each element in $\mathcal{T}_m$ tends to zero as $m \to +\infty$ (see \cite{chen2004optimal} for details). In \cite{chen2004optimal}, Chen and Xu  considered the approximation 
\begin{align*}
    \triangle_p(v, P_{m})=\min\left\{\int_{B} |v(x)-l_m(x)|^{p}\,\mathrm{d}x: l_m\in P_{m}\right\}, \qquad p\geq 1.
\end{align*}
They proved that if $v: B\subset \mathbb{R}^n\to \mathbb{R}$ is a convex function of class $C^2_+(B)$, then there exists a constant $\bar{\delta}_{p,n}$, depending only on $n$ and $p$, such that  
\begin{align*}
      \triangle_p(v, P_{m}) \sim
     \bar{\delta}_{p,n}\left(\int_{B} (\det \D v(x))^{\frac{p}{n+2p}}\,\mathrm{d}x\right)^{\frac{n+2p}{n}}\frac{1}{m^{\frac{2p}{n}}}.    
\end{align*}
In \cite{chen2007optimal}, Chen, Sun, and Xu obtained an analogous result for $p = +\infty$.   Babenko et al.~\cite{babenko2014exact}  derived the exact asymptotics of the optimal weighted error for linear spline interpolation of $C^2_+$ functions in $\mathbb{R}^2$, for $0 < p < +\infty$. In a related work, Babenko~\cite{babenko2010exact} studied the uniform interpolation error of $C^2_+$ functions by multilinear splines, i.e., the maximum pointwise deviation between a function and its interpolant over a uniform rectangular partition of a cube (or rectangular domain), and derived its exact asymptotic behavior as the mesh is refined.

% Babenko et al. \cite{babenko2014exact} derived the exact asymptotics of the optimal weighted error for linear spline interpolation of $C^2_+$ functions in $\mathbb{R}^2$, for $0 < p < +\infty$, while in \cite{babenko2010exact}, Babenko considered the exact asymptotics of the uniform interpolation error by multilinear splines.

We aim to extend the theory of polytopal approximation from convex bodies to convex functions with compact domain. Let $u:\mathbb{R}^n\to (-\infty,+\infty]$ be a convex function with compact domain. We consider the class 
\[
P_{(m)}^c = P_{(m)}^c(u)
\]
consisting of piecewise affine functions $l_m$ circumscribed to $u$, i.e., satisfying   $l_m(x)~\le~u(x)$ for all $x \in \dom u$, and which can be written as the maximum of at most $m$ affine functions, that is, there exist finitely many affine functions $\psi_1, \dots, \psi_m: \mathbb{R}^n\to \mathbb{R}$ such that
\begin{align*}
 l_m(x)= \max_{1 \le j \le m} \psi_j(x), \quad  x \in \mathbb{R}^n.
\end{align*}
Let $\omega : \mathbb{R}^{n+1} \to \mathbb{R}$ be a continuous positive  function. Our goal is to estimate the weighted approximation 
\begin{align}\label{minimum}
    \triangle_p(u, P_{(m)}^c, \omega) 
    = \min \Bigl\{ \int_{\dom u} (u(x) - l_m(x))^p \, \omega(x,u(x)) \, \mathrm{d}x : l_m \in P_{(m)}^c \Bigr\},
\end{align}
for sufficiently large $m$ and $p>0$.

This problem can be seen as a functional analogue of classical results on the approximation of convex bodies by circumscribed polytopes (cf. \eqref{gruber}). In contrast to interpolation methods based on inscribed triangulations of the domain, the approximating functions considered here are defined as the maximum of a fixed number of circumscribed affine functions.  This perspective not only extends Gruber’s approach \cite[Theorem~11.4]{gruber2007convex} to the functional setting, but also establishes a natural connection with the $\zeta$-affine surface area functional $Z_\zeta(u)$ introduced in \eqref{affine}, since the asymptotic behavior of $\triangle_p$ is governed by a power of the determinant of the Hessian of $u$, weighted appropriately.  We also analyze   the weighted approximation in the dual space of $\LC(\mathbb{R}^{n})$ via the Legendre transform. In particular, for the weight $\omega(x,t) = e^{-t}$, $(x,t) \in \mathbb{R}^{n+1}$, the asymptotic behavior is closely related to the \emph{weighted functional affine surface area} introduced in \cite{SchuttThaeleTurchiWerner2024}.  
The precise statement is discussed in Section~\ref{dc}.

% Following the approach of Gruber in \cite[Theorem 11.4]{gruber2007convex}
\begin{teo}\label{teoprin}
Let $u \in \LC\RR$ be of class $C^2_+$ in $\interior (\dom u)$, and let $\omega : \mathbb{R}^{n+1} \to \mathbb{R}$ be continuous and positive.  
Then there exists a constant $\delta_{p,n}$, depending only on $n$ and $p>0$, such that
\begin{align*}
    \lim_{m \to +\infty} m^{\frac{2p}{n}} \triangle_p(u, P_{(m)}^c, \omega) 
    = \frac{\delta_{p,n}}{2^p} \left( \int_{\dom u} (\det \D u(x))^{\frac{p}{n+2p}} \, \omega(x,u(x))^{\frac{n}{n+2p}} \, \mathrm{d}x \right)^{\frac{n+2p}{n}}.
\end{align*}
\end{teo}
\noindent The constant $\delta_{p,n}$ is given by Zador's Theorem (see Section \ref{tools}), and for $p=1$ it coincides with the constant given in \eqref{gruber}. 

%We will see in Section \ref{dc} that for $\omega(x,u(x))=e^{-u(x)}$ this result is related in some sense with the \emph{weighted functional affine surface area} introduced in  \cite{SchuttThaeleTurchiWerner2024}.
\medskip

In particular, by choosing the weight function $\omega \equiv 1$, Theorem \ref{teoprin} immediately yields the corresponding asymptotic result for the unweighted case.  

\begin{cor}
Let $u \in \LC(\mathbb{R}^n)$ be of class $C^2_+$ in $\interior(\dom u)$.  
Then  there exists a constant $\delta_{p,n}$, depending only on $n$ and $p>0$, such that
\begin{align*}
    \lim_{m \to +\infty} m^{\frac{2p}{n}} \triangle_p(u, P_{(m)}^c, 1) 
    = \frac{\delta_{p,n}}{2^p} \left( \int_{\dom u} (\det \D u(x))^{\frac{p}{n+2p}} \, \mathrm{d}x \right)^{\frac{n+2p}{n}}.
\end{align*}
\end{cor}

For $p=1$, we obtain the  affine surface area  that corresponds to the graph of the convex function $u\in \LC(\mathbb{R}^n)$.

\section{Preliminaries}\label{tools}
In this section, we collect background results that will be needed throughout the paper. 
\paragraph{Zador’s Theorem.}
We begin with the constant $\delta_{n}$ appearing in inequality \eqref{gruber}, which is connected to Zador’s Theorem for $\alpha = 1$.
\begin{teo}[\cite{zador1982asymptotic}, Zador's Theorem]\label{zt}
Let $\alpha > 0$. Then there exists a constant $ \delta_{2\alpha, n} > 0$, depending only on $\alpha$ and $n$, such that for any Jordan measurable compact set $J \subset \mathbb{R}^n$ with $V_n(J) > 0$,
\begin{align*}
    \lim_{m \to +\infty} m^{\frac{2\alpha}{n}} 
    \inf_{\substack{S \subset J \\ \# S = m}} 
    \int_J \min_{s \in S} \|x-s\|^{2\alpha} \, \mathrm{d}x 
    = \delta_{2\alpha, n} \, V_n(J)^{\frac{n+2\alpha}{n}}.
\end{align*}
\end{teo}
\noindent Here, $\#S$ denotes the cardinality of the set $S$, and $V_n(A)$ denotes the $n$-dimensional volume of a set $A\subset \mathbb{R}^n$.
A set $J \subset \mathbb{R}^n$ is called \emph{Jordan measurable} if its characteristic function is Riemann integrable.

From Theorem \ref{zt}, we obtain the following result.

\begin{lema}\label{lema5.1}
Let $J \subset \mathbb{R}^n$ be Jordan measurable with $V_n(J) > 0$, and let $q$ be a positive definite quadratic form on $\mathbb{R}^n$. Then
\begin{align*}
    \lim_{m \to +\infty} m^{\frac{2p}{n}} 
    \inf_{\substack{S \subset J \\ \# S = m}} 
    \int_J \min_{s \in S} |q(x-s)|^p \, \mathrm{d}x
    = \delta_{p,n} \, V_n(J)^{\frac{n+2p}{n}} (\det q)^{\frac{p}{n}}.  
\end{align*}
\end{lema}

By applying Zador's Theorem with $\alpha=p$, we obtain the constant that appears in Theorem \ref{teoprin}.

\paragraph{Convex functions.}
We denote by 
$$\Conv(\mathbb{R}^n)=\{u: \mathbb{R}^n\rightarrow (-\infty, +\infty]\mid \ u \  \mbox{is l.s.c and convex}, \ u\not\equiv +\infty\} $$
the space of  proper, lower semicontinuous, convex functions defined on $\mathbb{R}^n$.

We  say that a sequence $u_k\in \Conv(\mathbb{R}^n)$ \textit{epi-converges} to $u$ if, for every sequence $x_k$ converging to $x$,  
$$u(x)\leq \liminf_{k\rightarrow +\infty} u_k(x_k),$$
and there exists a sequence $x_k$ converging to $x$ such that 
$$u(x)= \lim_{k\rightarrow +\infty}u_k(x_k).$$

The following theorem provides a concrete characterization of epi-convergence, showing its equivalence with a more familiar notion of convergence.  

\begin{teo}[\cite{rockafellar2009variational}, Theorem 7.17]\label{equivalence}
Let $u_k\in \Conv(\mathbb{R}^n)$ be a sequence of  functions which epi-converges to $u$. If $\dom u$ has non-empty interior, the following statements are equivalents
\begin{itemize}
    \item[(1)] $u_k$ epi-converges to $u$;
    %\item[(2)] there is a dense subset $D$ of $\, \mathbb{R}^n$ such that $u_k(x)$ pointwise converges to $u(x)$, for all $x\in D$.
    \item[(2)] $u_k$ converges uniformly to $u$ on every compact set $C$ that does not
     contain a boundary point of $\dom u$.
\end{itemize}
\end{teo}

We are interested in exploring a dual perspective related to Theorem~\ref{teoprin}.  
To this end, it is necessary to introduce the Monge--Ampère measure and the dual space of $\LC(\mathbb{R}^n)$ via the Legendre transform.

Let $u\in\Conv(\mathbb{R}^n)$. The \emph{subdifferential} of $u$ at $x \in \dom u$ is defined as 
$$\partial u(x)=\{y\in\mathbb{R}^n: u(z)\geq u(x)+ y \cdot (z-x) \ \mbox{for all } \ z\in \mathbb{R}^n\},$$
and we set $\partial u(x)=\emptyset$ for $x\not\in\dom u$,  where 
$x \cdot y = \sum_{i=1}^n x_i y_i$ denotes the standard inner product in $\mathbb{R}^n$. For a subset $B \subset \mathbb{R}^n$, the image of $B$ under the subdifferential of $u$ is
$$\partial u(B)= \bigcup_{x\in B}\partial u(x).$$
The Monge--Ampère measure of $u$ is then defined, for every Borel set $B \subseteq \dom u$, by
\begin{align}\label{mong}
\MA(u; B) := V_n(\partial u(B)).   
\end{align}
Moreover, by \cite[Theorem~A.31]{figalli2017monge}, if $B\subset\dom u$ is an open bounded set and  $u \in C^2(B)$, then
\begin{align}\label{maa}
    \MA(u; B) = \int_{B} \det \D u(x) \, \mathrm{d}x.
\end{align}

For a convex body $K \in \mathcal{K}^{n+1}$, the \emph{support function} is defined by
\[
h_K(x) = \max_{y \in K}  x\cdot  y, \qquad x \in \mathbb{R}^{n+1}.
\]
If $\mu$ is a measure on $\mathbb{R}^n$, we write $\supp(\mu)$ for the support of $\mu$. The following class of convex functions will be considered
\begin{multline}\label{def_dual}
 \Conv_{\MA}\finite 
\\[1mm]
    = \Bigl\{ h_K(\cdot, -1) \mid \ K \in \mathcal{K}^{n+1}, \ 
    \supp\bigl(\MA(h_K(\cdot, -1); \cdot)\bigr) \text{ is compact in } \mathbb{R}^n \Bigr\}.
\end{multline}
Note that 
\begin{align}\label{lip_monge}
\Conv_{\MA}\finite \subset \lip\finite,    
\end{align}
the space of finite-valued Lipschitz convex functions on $\mathbb{R}^n$ (see, e.g., \cite[equation (8)]{baeta_ludwig_semicontinuity}).

We equip $\Conv_{\MA}\finite$ with the following notion of convergence: we say that  $v_k \in \Conv_{\MA}\finite$ is  \emph{$\tau^*$}-convergent  to $v \in \Conv_{\MA}\finite$ if the following conditions hold
\begin{itemize}
    \item[(i)] $v_k$ epi-converges to $v$;
    \item[(ii)] there exists a compact set $C \subset \mathbb{R}^n$ containing the supports of $\MA(v; \cdot)$ and $\MA(v_k; \cdot)$ for all $k \in \mathbb{N}$.
\end{itemize}

By \cite[Lemma 3]{baeta_ludwig_semicontinuity}, we have the following duality via Legendre transform
\begin{align*}
    (\Conv_{\MA}\finite)^*= \LC\RR,
\end{align*}
where the  \textit{Legendre transform}  of  a convex function $u: \mathbb{R}^n\to (-\infty, +\infty]$ is given by 
\begin{align*}
    u^*(x)=\sup_{y\in\mathbb{R}^n} (\langle x,y \rangle - u(y)), \hspace{0.1cm} x \in\mathbb{R}^n.
\end{align*}
In particular, this duality justifies the terminology of \(\tau^*\)-convergence. 
The following lemma makes this precise.

\begin{lema}[\cite{baeta_ludwig_semicontinuity}, Lemma~3]\label{duality}
A sequence $u_k$ in $\LC\RR$ is $\tau$-convergent to $u \in \LC\RR$ if and only if $u_k^*$ and $u^*$ belong to 
$\Conv_{\MA}\finite$ and $u_k^*$ is $\tau^*$-convergent to  $u^*$.
\end{lema}

\section{An Equivalent Framework and Applications}\label{dc}
Consider the class
\[
\mathcal{F}(X) = \{ v: X \to \mathbb{R} \mid v \text{ is Lipschitz, convex, and of class } C^2_+ \text{ in } \interior(X) \},
\]
where $X \subset \mathbb{R}^n$ is compact. We equip $\mathcal{F}(X)$ with the topology of uniform convergence. Let $\omega: \mathbb{R}^{n+1} \to \mathbb{R}$ be a continuous positive  function. As in \eqref{minimum}, for $v \in \mathcal{F}(X)$ we define the weighted approximation
\begin{align*}
    \triangle_p(v, P_{(m)}^c, \omega) 
    := \min \Bigl\{ \int_X (v(x) - l_m(x))^p \, \omega(x, v(x)) \, \mathrm{d}x : l_m \in P_{(m)}^c \Bigr\},
\end{align*}
for sufficiently large $m$ and $p>0$.
\medskip

Theorem \ref{teoprin} is a consequence of the following theorem (see Section \ref{proof}).
\begin{teo}\label{main2}
Let $v \in \mathcal{F}(X)$ and let $\omega: \mathbb{R}^{n+1} \to \mathbb{R}$ be continuous and positive. Then there exists a constant $\delta_{p,n}$, depending only on $n$ and $p>0$, such that
    \begin{align*}
     \lim_{m\rightarrow +\infty} m^{\frac{2p}{n}}\triangle_p(v, P_{(m)}^c,\omega)=\frac{\delta_{p,n}}{2^p}\left(\int_{X} (\det \D v(x))^{\frac{p}{n+2p}}\omega(x,v(x))^\frac{n}{n+2p}\dif x\right)^{\frac{n+2p}{n}}.    
    \end{align*}
\end{teo}
\noindent The constant $\delta_{p,n}$ is given by Zador's Theorem  with $\alpha=p$.
In order to establish Theorem \ref{main2}, we proceed along the lines of the method developed in \cite[ Theorem 11.4]{gruber2007convex}.

The Lipschitz condition in the definition of $\mathcal{F}(X)$ ensures that the integral
\begin{align*}
 \int_{X} (\det \D v(x))^{\frac{p}{n+2p}}\omega(x,v(x))^\frac{n}{n+2p}\dif x   
\end{align*}
is finite for every $p>0$. Indeed, since $\omega(\cdot, v(\cdot))$ is continuous and positive on the compact set $X$, it attains a maximum, say  $\omega_{\max} = \max_{x \in X} \omega(x,v(x))>0$. Consequently,
\begin{align*}
\int_{X} (\det \D v(x))^{\frac{p}{n+2p}}\omega(x,v(x))^\frac{n}{n+2p}\dif x& \leq  \omega_{\max}^\frac{n}{n+2p}  \int_{X} (\det \D v(x))^{\frac{p}{n+2p}}\dif x. 
\end{align*}
The finiteness of the right-hand side follows from the fact that if $L$ is the Lipschitz constant of $v$, then $\|y\|\leq L$ for every $y\in\partial v(x)$ and $x\in \interior (\dom u)$ (see, for example, \cite[Theorem 3.61]{beck2017}), together with Jensen's inequality,  the definition of the Monge--Ampère measure \eqref{mong}, and  \eqref{maa}, see also \cite[Theorem 1]{baeta_ludwig_semicontinuity}.

Consider the functional
\begin{align}\label{dual_a}
  Z_{\tilde{\zeta}}^*(v)=\int_{\mathbb{R}^n} \tilde{\zeta}(\det \D v(x))\dif x
\end{align}
for $v\in \Conv_{\MA}\finite$ and $\tilde{\zeta}\in \conc$. In \cite{baeta_ludwig_semicontinuity}, it is shown that $Z_{\tilde{\zeta}}^*$ is related to the $\zeta$-affine surface area functional  \eqref{affine} in the following way
\begin{align*}
   Z_{\tilde{\zeta}}^*(v)= Z_{\zeta}(v^*)
\end{align*}
where $\zeta(t)= \tilde{\zeta}(1/t)\, t$ for $t> 0$. In particular, $\zeta\in \conc$.  Consequently, the functional \eqref{dual_a} is also  naturally connected to the affine surface area of the graph of a convex function.

%We now introduce an extended class of piecewise affine functions for convex functions in $ \Conv_{\MA}\finite$. Let
%\[
%P_{(m)}^{\ext}(v) 
%\]
%denote the set of functions $l_m: \mathbb{R}^n \to \mathbb{R}$ satisfying the following conditions: 
%\begin{itemize}
%    \item[(1)] $l_m(x) = \max_{1 \le j \le m} w_j(x)$ for some affine functions $w_j$ and $x\in \supp(\MA(v;\cdot))$;
%    \item[(2)] $v(x) \ge l_m(x)$ for every $x \in \mathbb{R}^n$;
%    \item[(3)] for $x \in \mathbb{R}^n \setminus \supp(\MA(v;\cdot))$, 
%    \[
%    0 \le m^{\frac{2}{n}}(v(x) - l_m(x))\leq o(m), \quad o(m)\to 0, \text{ as } m\to +\infty.
%    \]
%\end{itemize}
%This definition ensures that $l_m$ is convex on all of $\mathbb{R}^n$, coincides with the classical definition on $X=\supp(\MA(v;\cdot))$, and approximates $v$ outside $X$ in a controlled way. 

%Initially, one might worry that the class $P_{(m)}^{\ext}(v)$ is empty due to the condition outside $X$.  However, since $v \in \Conv_{\MA}\finite$, its Monge--Ampère measure is supported on a compact set $X \subset \mathbb{R}^n$. 

%Therefore, it is always possible to construct functions $l_m \in P_{(m)}^{\ext}(v)$ that satisfy all the conditions. 

For a function $v \in \Conv_{\MA}\finite$, we denote by
\[
P_{(m)}^c\big(v\big|_{\supp(\MA(v;\cdot))}\big)
\] 
the class of piecewise affine functions $l_m$ that are circumscribed to $v$ on $\supp(\MA(v;\cdot))$, i.e.,
\[
l_m(x) \le v(x) \quad \text{for all } x \in \supp(\MA(v;\cdot)).
\] 
Moreover, each $l_m \in P_{(m)}^c\big(v\big|_{\supp(\MA(v;\cdot))}\big)$ can be represented as the maximum of at most $m$ affine functions.

\begin{teo}\label{ma}
Let $v \in \Conv_{\MA}\finite$ be of class $C^2_+$ on 
$\supp(\MA(v;\cdot))$, and let $\omega:\mathbb{R}^{n+1} \to \mathbb{R}$ be continuous and   positive. Then
\begin{align*}
\lim_{m\to+\infty} m^{\frac{2p}{n}} \min \Bigl\{ \int_{ \supp(\MA(v;\cdot))} (v(x)-l_m(x))^p \, \omega(x,v(x)) \, \mathrm{d}x : l_m \in P_{(m)}^c\big(v\big|_{\supp(\MA(v;\cdot))}\big)  \Bigr\} \\
= \frac{\delta_{p,n}}{2^p} \Biggl( \int_{\mathbb{R}^n} (\det \D v(x))^{\frac{p}{n+2p}} \, \omega(x,v(x))^{\frac{n}{n+2p}} \, \mathrm{d}x \Biggr)^{\frac{n+2p}{n}}.
\end{align*}
\end{teo}

For the special case where $p=1$ and $\omega(x,v(x))= e^{- v(x)}$, we obtain the following.
\begin{cor}
Let $v \in \Conv_{\MA}\finite$ be of class $C^2_+$ on 
$\supp(\MA(v;\cdot))$.
Then
\begin{multline*}
\lim_{m \to +\infty} m^{\frac{2}{n}} \min \Bigl\{ \int_{\supp(\MA(v;\cdot))} (v(x)-l_m(x)) \, e^{-v(x)} \, \mathrm{d}x : l_m \in P_{(m)}^c\big(v\big|_{\supp(\MA(v;\cdot))}\big)  \Bigr\} \\
= \frac{\delta_{p,n}}{2^p} \Biggl( \int_{\mathbb{R}^n} (\det \D v(x))^{\frac{1}{n+2}} \, e^{-\frac{n}{n+2} v(x)} \, \mathrm{d}x \Biggr)^{\frac{n+2}{n}}.
\end{multline*}    
\end{cor}

The functional
\begin{align*}
    Z(v)= \int_{\mathbb{R}^n} (\det \D v(x))^{\frac{1}{n+2}} \, e^{-\frac{n}{n+2} v(x)} \, \mathrm{d}x
\end{align*}
coincides with the \emph{weighted functional affine surface area} introduced in 
\cite{SchuttThaeleTurchiWerner2024}.  The class of convex functions considered in \cite{SchuttThaeleTurchiWerner2024} is not identical to the one considered in this work, although there is a nontrivial overlap between the two settings. Moreover, by \cite[Theorem~9]{baeta_ludwig_semicontinuity}, this functional is finite and $\tau^*$-upper semicontinuous.

\section{Proof of Theorem \ref{main2}}\label{teo15}
Let $\lambda > 1$ and let $v \in \mathcal{F}(X)$. Fix a point $a \in \interior(X)$. 
By Taylor’s formula, there exists $\xi = \xi(x)\in (0,1)$ such that  
\begin{align}\label{Taylor}
    v(x) = v(a) + \nabla v(a)\cdot (x-a) 
    + \frac{1}{2}(x-a)\cdot  \, \D v\!\left(a+\xi(x-a)\right)  (x-a),
\end{align}
where  $x\in X$. For each $x \in \interior(X)$, define the quadratic form
\begin{align}\label{defi}
    q_x(y) := y\cdot \D v(x)\, y, 
    \qquad y \in \mathbb{R}^n.
\end{align}
Note that $\det(\D q_x) = \det(\D v(x))$ for all 
$x \in \interior(X)$.

Since $v \in C^2_+(X)$, the coefficients of $q_x$ vary continuously with $x$. 
Hence, there exists an open convex neighborhood 
$U \subset X$ of $a$ such that, for all 
$x \in U$ and $y \in \mathbb{R}^n$, 
\begin{align}\label{eq6.6}
    \frac{1}{\lambda} \, q_a(y) \;\leq\; q_x(y) \;\leq\; \lambda \, q_a(y).
\end{align}
Moreover, for all $x \in U$ we have 
\begin{align}\label{eq6.7}
    \frac{1}{\lambda^n}\, \det(\D q_a) 
    \;\leq\; \det(\D q_x) 
    \;\leq\; \lambda^n \, \det(\D q_a).
\end{align}
For simplicity, we will denote $\det(\D q_x)$ simply by 
$\det(q_x)$ for every $x \in \operatorname{int}(X)$.

We begin by proving the following inequality
\begin{align}\label{eq6.4}
 m^{\frac{2p}{n}}\triangle_p(v, P_{(m)}^c,\omega)\geq \frac{\delta_{p,n}}{2^\frac{1}{p}\lambda^{\frac{3n+2np+2p}{n}}}\left(\int_{X}(\det \D v(x))^\frac{p}{n+2p}\omega(x,v(x))^\frac{n}{n+2p}\dif x\right)^\frac{n+2p}{n}   
\end{align}
for all sufficiently large $m$.

The open neighborhoods $U$ satisfying \eqref{eq6.6} and \eqref{eq6.7} form a cover of $X$.  Therefore, we can extract a finite subcover. By Lebesgue's covering lemma 
(see, e.g., \cite[p. 154]{kelley1955general}), each set in $X$ with sufficiently small diameter  is contained in one of the neighborhoods from the subcover.  This leads to the following result.

\begin{prop}\label{prop10}
Let $\lambda > 1$ and let $v \in \mathcal{F}(X)$. There exist small pieces $J_i$, $i=1,\dots, l$, in $X$, points $a_i$ and neighborhoods $U_i$ of $a_i$ in $X$ such that the sets 
$J_i \subseteq U_i \subseteq \interior(X)$ are compact, pairwise disjoint, and Jordan measurable, and
\begin{align}\label{eq6.6.}
 & \frac{1}{\lambda} q_{a_i}(y) \leq q_x(y) \leq \lambda q_{a_i}(y), 
 \qquad y\in\mathbb{R}^n, \ x\in U_i,
\\ & \label{eq6.7.}
 \frac{1}{\lambda^n} \det q_{a_i} \leq \det q_x \leq \lambda^n \det q_{a_i}, 
 \qquad x\in U_i,
\\ & \label{omega_lambda}
 \frac{1}{\lambda} \omega(x,v(x)) \leq \omega_{a_i} \leq \lambda \, \omega(x,v(x)), 
 \qquad x \in U_i,
\end{align}
and
\begin{align}\label{eq6.9}
 \sum_{i=1}^l \int_{J_i} (\det \D v(x))^\frac{p}{n+2p} \, \omega(x,v(x))^\frac{n}{n+2p} \, \dif x
 \geq \frac{1}{\lambda} \int_X (\det \D v(x))^\frac{p}{n+2p} \, \omega(x,v(x))^\frac{n}{n+2p} \, \dif x.
\end{align}
\end{prop}
\noindent Note that conditions \eqref{omega_lambda} and \eqref{eq6.9} follow from the continuity of $\omega(\cdot, v(\cdot))$ and the fact that $v$ is of class $C^2_+$ in $X$. From now on, we assume that the sets $J_i$ and $U_i$, $i=1,\dots, l$, are those given by Proposition \ref{prop10}.

Let 
\[
\sigma_m = \bigl( (\gamma_1,\beta_1), \dots, (\gamma_m,\beta_m) \bigr) \in (\mathbb{R}^n \times \mathbb{R})^m,
\]
where each pair $(\gamma_j,\beta_j)\in \mathbb{R}^n \times \mathbb{R}$ corresponds to the coefficients of an affine function 
\[
x \mapsto \gamma_j \cdot x + \beta_j, \quad j=1,\dots,m.
\]
Define the function
\begin{align*}
   \Phi(\sigma_m) \;=\;  \int_X \left|v(x) - \max_{1\le j \le m} (\gamma_j \cdot x + \beta_j)\right|^p \, \omega(x,v(x)) \, \dif x, \qquad \sigma_m \in (\mathbb{R}^n \times \mathbb{R})^m.
\end{align*}
Then the best approximation  can be written as
\begin{multline*}
   \triangle_p(v, P_{(m)}^c, \omega) 
    = \min \Bigl\{ \int_{X} (v(x) - l_m(x))^p \, \omega(x,v(x)) \, \mathrm{d}x : l_m \in P_{(m)}^c \Bigr\}\\
    = \min \Bigl\{ \Phi((\gamma_1,\beta_1), \dots, (\gamma_m,\beta_m)):  v(x)\geq \max_{1\le j \le m} (\gamma_j \cdot x + \beta_j)(x) \text{ for all } x\in X\Bigr\}, 
\end{multline*}
and since $\Phi(\sigma_m)$ is continuous on $(\mathbb{R}^n \times \mathbb{R})^m$ and $X$ is compact,  this minimum is attained for some $\sigma_m=\bigl( (\gamma_1,\beta_1), \dots, (\gamma_m,\beta_m) \bigr)$.

Let $l_m \in P_{(m)}^c$ for $m = n+2, n+3, \dots$ be a sequence of best approximating convex piecewise affine functions of $v$.
Since $l_m$ converges uniformly to $v$, we have 
\begin{align}\label{eq6.1}
\int_X (v(x) - l_m(x))^p \, \omega(x,v(x)) \, \dif x \longrightarrow 0 \quad \text{as } m \to +\infty.
\end{align}
Let $w_1, \dots, w_m$ be the affine functions that define $l_m$, i.e.,
\begin{align}\label{piece}
    l_m(x) = \max_{1\leq j \leq m} \psi_j(x).
\end{align}
Note that there exist sets $C_1, \dots, C_m \subset X$ such that $X\subseteq \bigcup_{j=1}^m C_j$ and 
\[
l_m(x) = \psi_j(x) \quad \text{for all } x \in C_j, \ j=1,\dots,m,
\]
where an affine function $\psi_j$ may appear in several sets $C_j$,  since $l_m$ is defined by at most $m$ affine functions.

We denote by $C_{i_k}^m$, $k = 1, \dots, d_{i}^m$, the subsets of the collection $\{C_1, \dots, C_m\}$ such that 
\[
C_{i_k}^m \cap J_i \neq \emptyset, \quad i = 1, \dots, l.
\]
By \eqref{eq6.1}, and since $v$ is strictly convex, it follows that the diameter of the sets $C_{i_k}^m$ tends to zero as $m \to +\infty$.  This, combined with Proposition \ref{prop10}, implies that
\begin{align}\label{eq6.10}
\int_X (v(x) - l_m(x))^p \, \omega(x,v(x)) \, \dif x
\geq \sum_{i=1}^l \int_{J_i} (v(x) - l_m(x))^p \, \omega(x,v(x)) \, \dif x,
\end{align}
and
\begin{align}\label{eq5.12}
d_{1}^m + \cdots + d_{l}^m \leq m,
\end{align}
for all sufficiently large $m$, as well as
\begin{align}\label{eq5.11}
d_{i}^m \to +\infty \quad \text{as } m \to +\infty, \quad i = 1,\dots,l.
\end{align}

Note that for an optimal piecewise affine function, each $\psi_i$ satisfying \eqref{piece} must agree with $v$ at some point.  
Let $a_{i_k}^m \in C_{i_k}^m$ be the point where $\psi_{i_k}$ coincides with  $v$. Then, 
\begin{align}\label{ex}
l_m(x)=\psi_{i_k}(x)=   v(a_{i_k}^m) + \nabla v(a_{i_k}^m) \cdot (x - a_{i_k}^m)   
\end{align}
for every $x\in C_{i_k}^m$. That is,
\begin{multline}\label{eq5.13}
 \int_{J_i} (v(x)-l_m(x))^p \, \omega(x,v(x)) \, \dif x \\
= \sum_{k=1}^{d_{i}^m} \int_{C_{i_k}^m \cap J_i} 
\big(v(x) - v(a_{i_k}^m) - \nabla v(a_{i_k}^m) \cdot (x - a_{i_k}^m) \big)^p \, 
\omega(x,v(x)) \, \dif x.
\end{multline}

According to \eqref{Taylor}, we have
\begin{align}\label{eq5.101}
 v(x) - v(a_{i_k}^m) - \nabla v(a_{i_k}^m) \cdot (x - a_{i_k}^m)
 = \frac{1}{2} q_{a_{i_k}^m + \xi(x - a_{i_k}^m)}(x - a_{i_k}^m),
\end{align}
for $x \in C_{i_k}^m \cap J_i$, with a suitable $\xi \in [0,1]$ depending on $x$.  Since the diameter of $C_{i_k}^m$ tends to zero as $m \to +\infty$, 
Proposition \ref{prop10} ensures that, for sufficiently large $m$, 
all sets $C_{i_k}^m$ intersecting $J_i$ are contained in $U_i$.  
For such $m$, it then follows from \eqref{eq6.6.} that
\begin{align}\label{eq5.100}
\frac{1}{2} q_{a_{i_k}^m + \xi(x - a_{i_k}^m)}(x - a_{i_k}^m)
\geq \frac{1}{2\lambda} q_{a_i}(x - a_{i_k}^m),
\qquad x \in C_{i_k}^m \cap J_i.
\end{align}

Thus, by \eqref{eq6.10}, \eqref{eq5.13}, \eqref{eq5.101}, and \eqref{eq5.100}, we obtain
\begin{align*}
\triangle_p(v, P_{(m)}^c,\omega)
&= \int_X (v(x) - l_m(x))^p \, \omega(x,v(x)) \, \dif x \\
&\geq \frac{1}{(2\lambda)^p} \sum_{i=1}^{l} \sum_{k=1}^{d_{i}^m} \int_{C_{i_k}^m \cap J_i} 
(q_{a_i}(x - a_{i_k}^m))^p \, \omega(x,v(x)) \, \dif x \\
&\geq \frac{1}{(2\lambda)^p} \sum_{i=1}^{l} \int_{J_i} 
\min_{k=1,\dots,d_{i}^m} (q_{a_i}(x - a_{i_k}^m))^p \, \omega(x,v(x)) \, \dif x \\
&\geq \frac{1}{(2\lambda)^p} \sum_{i=1}^{l} 
\inf_{\substack{S \subset \mathbb{R}^n \\ \#S = d_{i}^m}} 
\int_{J_i} \min_{t \in S} (q_{a_i}(x - t))^p \, \omega(x,v(x)) \, \dif x,
\end{align*}
from which, using Lemma \ref{lema5.1} together with \eqref{omega_lambda} and \eqref{eq5.11}, it follows that
\begin{align}\label{eq5.102}
\triangle_p(v, P_{(m)}^c,\omega) 
\geq \frac{\delta_{p,n}}{2^p \lambda^{p+2}} \sum_{i=1}^{l} 
V_n(J_i)^{\frac{n+2p}{n}} \, (\det q_{a_i})^{\frac{p}{n}} \, \omega_{a_i} \, \frac{1}{(d_{i}^m)^{\frac{2p}{n}}}.
\end{align}

Next, we use a well-known consequence of Jensen's inequality
\begin{align}\label{inequality}
\frac{1}{m} \sum_{i=1}^m b_i^t \geq \left( \frac{1}{m} \sum_{i=1}^m b_i \right)^t, 
\end{align}
which holds for $t \geq 1$ and $b_i > 0$ for $i=1,\dots,m$.   Applying \eqref{eq5.102} and \eqref{inequality}, we obtain
\begin{align*} 
& \triangle_p(v, P_{(m)}^c,\omega) \\
&\geq \frac{\delta_{p,n}}{2^p \lambda^{p+2}} 
\left( \frac{1}{d_{1}^m+\cdots + d_{l}^m} 
\sum_{i=1}^{l} 
\big(V_n(J_i) (\det q_{a_i})^{\frac{p}{n+2p}} \omega_{a_i}^{\frac{n}{n+2p}} \frac{1}{d_{i}^m}\big)^{\frac{n+2p}{n}} d_{i}^m 
\right) \\
&\quad \times (d_{1}^m+\cdots + d_{l}^m) \\
&\geq \frac{\delta_{p,n}}{2^p \lambda^{p+2}} 
\left( \frac{1}{d_{1}^m+\cdots + d_{l}^m} 
\sum_{i=1}^{l} V_n(J_i) (\det q_{a_i})^{\frac{p}{n+2p}} \omega_{a_i}^{\frac{n}{n+2p}} \right)^{\frac{n+2p}{n}} 
\times (d_{1}^m+\cdots + d_{l}^m) \\
&= \frac{\delta_{p,n}}{2^p \lambda^{p+2}} 
\left( \sum_{i=1}^{l} V_n(J_i) (\det q_{a_i})^{\frac{p}{n+2p}} \omega_{a_i}^{\frac{n}{n+2p}} \right)^{\frac{n+2p}{n}} 
\frac{1}{(d_{1}^m+\cdots + d_{l}^m)^{\frac{2p}{n}}}.
\end{align*}
Using \eqref{eq6.7.}, \eqref{eq5.12}, \eqref{defi}, and \eqref{eq6.9}, we conclude \eqref{eq6.4}, since
\begin{align*}
\triangle_p(v, P_{(m)}^c,\omega) 
&= \int_X (v(x) - l_m(x))^p \, \omega(x,v(x)) \, \dif x \\
&\geq \frac{\delta_{p,n}}{2^p \lambda^{\,p+2+p+1}} 
\left( \sum_{i=1}^l \int_{J_i} (\det q_x)^{\frac{p}{n+2p}} \, \omega(x,v(x))^{\frac{n}{n+2p}} \, \dif x \right)^{\frac{n+2p}{n}} 
\frac{1}{m^{\frac{2p}{n}}} \\
&\geq \frac{\delta_{p,n}}{2^p \lambda^{\,2p+3+\frac{n+2p}{n}}} 
\left( \int_X (\det \D v(x))^{\frac{p}{n+2p}} \, \omega(x,v(x))^{\frac{n}{n+2p}} \, \dif x \right)^{\frac{n+2p}{n}} 
\frac{1}{m^{\frac{2p}{n}}} \\
&= \frac{\delta_{p,n}}{2^p \lambda^{\frac{3n+2np+2p}{n}}} 
\left( \int_X (\det \D v(x))^{\frac{p}{n+2p}} \, \omega(x,v(x))^{\frac{n}{n+2p}} \, \dif x \right)^{\frac{n+2p}{n}} 
\frac{1}{m^{\frac{2p}{n}}}
\end{align*}
for all sufficiently large $m$.
\goodbreak

Next, we will show that
\begin{align}\label{eq5.14}
m^{\frac{2p}{n}} \, \triangle_p(v, P_{(m)}^c,\omega) 
\leq \frac{\lambda^{\frac{4n+2np+4p}{n}} \, \delta_{p,n}}{2^p} 
\left( \int_X (\det \D v(x))^{\frac{p}{n+2p}} \, \omega(x,v(x))^{\frac{n}{n+2p}} \, \dif x \right)^{\frac{n+2p}{n}}   
\end{align}
for all sufficiently large $m$.

Using similar arguments to those at the beginning of the proof of \eqref{eq6.4}, we can partition $X$ into finitely many sets $D_i$, $i=1,\dots,l$, and select slightly larger sets $L_i$ such that $D_i \subseteq L_i \subset X$, along with points $a_i$ and neighborhoods $\tilde{U}_i$ in $X$ where the following result holds.
\begin{prop}\label{prop11}
Let $\lambda > 1$ and let $v \in \mathcal{F}(X)$. There exist small compact sets $D_i$ which partition $X$, with 
$D_i \subseteq L_i \subseteq \interior(\tilde{U}_i)$, $i=1,\dots, l$,
where $L_i$ is a Jordan measurable open set, $\tilde{U}_i$ is convex, and 
\begin{align}\label{eq5.16}
V_n(L_i) \leq \lambda \, V_n(D_i),
\end{align}    
\begin{align*}
& \frac{1}{\lambda} q_{a_i}(y) \leq q_x(y) \leq \lambda \, q_{a_i}(y), 
\quad y \in \mathbb{R}^n, \ x \in \tilde{U}_i,
\\ 
& \frac{1}{\lambda^n} \det q_{a_i} \leq \det q_x \leq \lambda^n \, \det q_{a_i}, 
\quad x \in \tilde{U}_i,
\end{align*}    
and 
\begin{align}\label{omega_lambda1}
\frac{1}{\lambda^n} \, \omega(x,v(x)) \leq \omega_{a_i} \leq \lambda^n \, \omega(x,v(x)), 
\quad x \in \tilde{U}_i.
\end{align}
\end{prop}
\noindent From now on, the sets $D_i$, $L_i$, and $\tilde{U}_i$, $i=1,\dots, l$, are those given by Proposition \ref{prop11}.
\goodbreak

We will construct piecewise affine functions $l_m$, each given by the maximum of at most $m$ affine functions, for sufficiently large $m$, such that the epigraph of $l_m$ contains the epigraph of $v$.  

Let
\begin{align}\label{eq5.18}
\tau_i = \frac{\int_{D_i} (\det \D v(x))^{\frac{p}{n+2p}} \, \omega(x,v(x))^{\frac{n}{n+2p}} \, \dif x}
{\int_X (\det \D v(x))^{\frac{p}{n+2p}} \, \omega(x,v(x))^{\frac{n}{n+2p}} \, \dif x}, 
\qquad 
d_{i}^m = \lfloor \tau_i m \rfloor,
\end{align}
where \(\lfloor \tau_i m \rfloor\) denotes the  greatest integer less than or equal to $\tau_i m$.

Then
\begin{align}
\label{eq5.19} & d_{i}^m \to +\infty, & &\text{as } m \to +\infty, \\
\label{eq5.20} & d_{i}^m \geq \frac{1}{\lambda} \, \tau_i m, & &\text{for all sufficiently large } m, \\
\label{eq5.21} & d_{1}^m + \cdots + d_{l}^m \leq m. 
\end{align}

Let $a_{i_k}^m \in L_i$, $k=1,\dots,d_{i}^m$, be chosen to minimize
\begin{multline*}
\int_{L_i} \min_{k=1,\dots,d_{i}^m} \big( q_{a_i}(x - a_{i_k}^m) \big)^p \, \omega(x,v(x)) \, \dif x \\
= \inf_{\substack{S \subseteq L_i \\ \#S = d_{i}^m}} 
\int_{L_i} \min_{t \in S} \big( q_{a_i}(x - t) \big)^p \, \omega(x,v(x))^{\frac{n}{n+2p}} \, \dif x.
\end{multline*}

Since $d_i^m \to +\infty$ as $m \to +\infty$ (cf.~\eqref{eq5.19}) and $L_i$ is compact,  the maximal distance from any point in $L_i$ to the nearest $a_{i_k}^m$ tends to zero.  In other words, for sufficiently large $m$, the points $a_{i_k}^m$ become densely distributed in each $L_i$, and therefore throughout $X $. Consider the affine functions tangent to $v$ at the points $a_{i_k}^m$.  
The function given by the maximum of these affine functions is a piecewise affine function, denoted by $l_m$,  and, by \eqref{eq5.21}, it is given by the maximum of at most $m$ affine functions. In addition, we have
\[
l_m \to v \quad \text{uniformly in } X.
\]
For all sufficiently large $m$, and as in \eqref{ex}, it also holds that
\begin{multline}\label{eq5.23}
\int_X (v(x) - l_m(x))^p \, \omega(x,v(x)) \, \dif x \\
\leq \sum_{i=1}^l \int_{L_i} 
\min_{k=1,\dots,d_{i}^m} \big( v(x) - v(a_{i_k}^m) - \nabla v(a_{i_k}^m) \cdot (x - a_{i_k}^m) \big)^p \, 
\omega(x,v(x)) \, \dif x.
\end{multline}
As shown in \eqref{eq5.101} and by Proposition \ref{prop11}, we have
\begin{align*}
v(x) - v(a_{i_k}^m) - \nabla v(a_{i_k}^m) \cdot (x - a_{i_k}^m) 
&= \frac{1}{2} q_{a_{i_k}^m + \xi(x - a_{i_k}^m)}(x - a_{i_k}^m) 
\\ &\leq \frac{\lambda}{2} q_{a_i}(x - a_{i_k}^m),
\end{align*}
with suitable $\xi \in (0,1)$ and $x \in L_i$.
\medskip

Combining this with \eqref{eq5.23}, \eqref{eq5.19}, Lemma \ref{lema5.1},  \eqref{omega_lambda1}, \eqref{eq5.16}, and \eqref{eq5.20}, we obtain
\begin{align*}
& \triangle_p(v, P_{(m)}^c,\omega) \\
& 
\leq \sum_{i=1}^l \int_{L_i} \min_{k=1,\dots,d_{i}^m} 
\big( v(x) - v(a_{i_k}^m) - \nabla v(a_{i_k}^m) \cdot (x - a_{i_k}^m) \big)^p \, \omega(x,v(x)) \, \dif x \\
&\leq \left( \frac{\lambda}{2} \right)^p \sum_{i=1}^l \int_{L_i} 
\min_{k=1,\dots,d_{i}^m} \big( q_{a_i}(x - a_{i_k}^m) \big)^p \, \omega(x,v(x)) \, \dif x \\
&\leq \frac{\lambda^{p+2} \, \delta_{p,n}}{2^p} \sum_{i=1}^l V_n(L_i)^{\frac{n+2p}{n}} (\det q_{a_i})^{\frac{p}{n}} \, \omega_{a_i} \frac{1}{(d_{i}^m)^{\frac{2p}{n}}} \\
&\leq \frac{\lambda^{p+2+\frac{n+2p}{n}+\frac{2p}{n}} \, \delta_{p,n}}{2^p} \sum_{i=1}^l 
V_n(D_i)^{\frac{n+2p}{n}} (\det q_{a_i})^{\frac{p}{n}} \frac{1}{\tau_i^{\frac{2p}{n}} m^{\frac{2p}{n}}}.
\end{align*}

Finally, using \eqref{eq5.14}, \eqref{inequality}, Proposition \ref{prop11}, \eqref{defi}, and \eqref{eq5.18}, we get
\begin{align*}
\triangle_\omega^p(v, P_{(m)}^c) 
&\leq \frac{\lambda^{\frac{3n+np+4p}{n}} \, \delta_{p,n}}{2^p} 
\sum_{i=1}^l \left( V_n(D_i) (\det q_{a_i})^{\frac{p}{n+2p}} \omega_{a_i}^{\frac{n}{n+2p}} \right)^{\frac{n+2p}{n}} \frac{1}{\tau_i^{\frac{2p}{n}} m^{\frac{2p}{n}}} \\
&\leq \frac{\lambda^{\frac{3n+np+4p+pn+n}{n}} \, \delta_{p,n}}{2^p} 
\sum_{i=1}^l \left( \int_{D_i} (\det q_x)^{\frac{p}{n+2p}} \omega(x,v(x))^{\frac{n}{n+2p}} \, \dif x \right)^{\frac{n+2p}{n}} \frac{1}{m^{\frac{2p}{n}}} \\
&\leq \frac{\lambda^{\frac{4n+2np+4p}{n}} \, \delta_{p,n}}{2^p} 
\left( \int_X (\det \D v(x))^{\frac{p}{n+2p}} \omega(x,v(x))^{\frac{n}{n+2p}} \, \dif x \right)^{\frac{n+2p}{n}} \frac{1}{m^{\frac{2p}{n}}}
\end{align*}
for all sufficiently large $m$.
\medskip

Therefore, \eqref{eq6.4} and \eqref{eq5.14} hold for any $\lambda > 1$, which concludes the proof of Theorem~\ref{main2}.

\section{Proof of Theorem \ref{teoprin} and Theorem \ref{ma}}\label{proof}
Recall that in Theorem \ref{teoprin} we  assume $u\in \LC(\mathbb{R}^n)$, i.e., $u:\mathbb{R}^n\to (-\infty, +\infty]$ is a Lipschitz convex function with compact domain (see \eqref{def_set}). In this space, a sequence $l_m$ is \emph{$\tau$-convergent} to $u$ if $l_m$ epi-converges to $u$ and the Lipschitz constants of the sequence are uniformly bounded.

\begin{proof}[Proof of Theorem \ref{teoprin}]
Let $u\in \LC\RR$. Then $\dom u$ is compact. If $\interior(\dom u) = \emptyset$, both integrals vanish and there is nothing to prove.   Hence, we may assume that $\dom u$ is full-dimensional. By Theorem \ref{equivalence}, if $l_m\in P_{(m)}^c$ converges uniformly to $u\in \LC\RR$ in $\interior (\dom u)$, then $l_m+\I_{\dom u}$ epi-converges to $u$. Moreover, since  the Lipschitz constants of $l_m$ are uniformly bounded,  we also have   that $l_m$ is $\tau$-convergent to $u$. Taking $X=\dom u$ in Theorem \ref{main2}, Theorem \ref{teoprin} immediately follows.     
\end{proof}

In Theorem \ref{ma}, the space of functions under consideration is $\Conv_{\MA}\finite$ as given by \eqref{def_dual}. By \eqref{lip_monge}, if $v\in \Conv_{\MA}\finite$, then $v$ is Lipschitz.
\medskip

\begin{proof}[Proof of Theorem \ref{ma}]
By Theorem \ref{main2}, we have 
\begin{multline}\label{case1}
\lim_{m\rightarrow +\infty} m^{\frac{2p}{n}}\min\left\{\int_{X} (v(x)-l_m(x))^{p}\omega(x,v(x))\dif x: l_m\in P_{(m)}^c\big(v|_{X}\big) \right\}\\
=  \frac{\delta_{p,n}}{2^p}\left(\int_{X} (\det \D v(x))^{\frac{p}{n+2p}}\omega(x,v(x))^\frac{n}{n+2p}\dif x\right)^{\frac{n+2p}{n}},  
\end{multline}    
where $X=\supp(\MA(v;\cdot ))$.  Moreover,  since $X$ is the support of the Monge--Ampère measure of $v$, by \eqref{maa} we have
\begin{align}\label{case2}
 \int_{\mathbb{R}^n \setminus X} (\det \D v(x))^{\frac{p}{n+2p}} \omega(x,v(x))^\frac{n}{n+2p} \, \mathrm{d}x = 0.   
\end{align}
Combining \eqref{case1} and \eqref{case2} completes the proof.
\end{proof}

\section*{Acknowledgments}
The author thanks Monika Ludwig for  insightful discussions and careful reading of the manuscript. The author also thanks Mohamed A. Mouamine for carefully reading the manuscript. This project was supported, in part, by the Austrian Science Fund (FWF) Grant-DOI: 10.55776/P34446 and Grant-DOI: 10.55776/P37030. For open access purposes, the author has applied a CC BY public copyright license to any author-accepted manuscript version arising from this submission.

\bibliography{ap}

\providecommand{\bysame}{\leavevmode\hbox to3em{\hrulefill}\thinspace}
\providecommand{\MR}{\relax\ifhmode\unskip\space\fi MR }
\providecommand{\MRhref}[2]{%
  \href{http://www.ams.org/mathscinet-getitem?mr=#1}{#2}
}
\providecommand{\href}[2]{#2}
\begin{thebibliography}{10}

\bibitem{aleksandrov1939second}
A.~D. Aleksandrov, \emph{Almost everywhere existence of the second differential
  of a convex function and some properties of convex surfaces connected with
  it}, Uchenye Zapiski Leningrad. Gos. Univ., Math. Ser. \textbf{6} (1939),
  3--35, in Russian.

\bibitem{babenko2014exact}
V.~Babenko, Y.~Babenko, N.~Parfinovych, and D.~Skorokhodov, \emph{Exact
  asymptotics of the optimal lp-error of asymmetric linear spline
  approximation}, Jaen J. Approx. \textbf{6} (2014), 1--36.

\bibitem{babenko2010exact}
Y.~Babenko, \emph{Exact asymptotics of the uniform error of interpolation by
  multilinear splines}, J. Approx. Theory \textbf{162} (2010), 1007--1024.

\bibitem{baeta_ludwig_semicontinuity}
F.~M. Baêta and M.~Ludwig, \emph{On the semicontinuity of functionals on
  function spaces}, 2025, \url{https://arxiv.org/abs/2509.17426}.

\bibitem{beck2017}
A.~Beck, \emph{First-order methods in optimization}, SIAM, Philadelphia, 2017.

\bibitem{boroczky2000approximation}
K.~Jr. Böröczky, \emph{Approximation of general smooth convex bodies}, Adv.
  Math. \textbf{153} (2000), 325--341.

\bibitem{chen2007optimal}
L.~Chen, P.~Sun, and J.~Xu, \emph{Optimal anisotropic meshes for minimizing
  interpolation errors in lp-norm}, Math. Comp. \textbf{76} (2007), 179--204.

\bibitem{chen2004optimal}
L.~Chen and J.~Xu, \emph{Optimal delaunay triangulations}, Journal of
  Computational Mathematics \textbf{22} (2004), no.~2, 299--308.

\bibitem{figalli2017monge}
A.~Figalli, \emph{The monge--ampère equation and its applications}, Zürich
  Lectures in Advanced Mathematics, European Mathematical Society (EMS),
  Zürich, 2017.

\bibitem{gruber1993asymptotic}
P.~M. Gruber, \emph{Asymptotic estimates for best and stepwise approximation of
  convex bodies ii}, Forum Math. \textbf{5} (1993), 521--538.

\bibitem{gruber2007convex}
\bysame, \emph{Convex and discrete geometry}, Grundlehren der Mathematischen
  Wissenschaften, vol. 336, Springer, Berlin, 2007.

\bibitem{kelley1955general}
J.~L. Kelley, \emph{General topology}, Van Nostrand, New York, 1955.

\bibitem{ludwig1999asymptotic}
M.~Ludwig, \emph{Asymptotic approximation of smooth convex bodies by general
  polytopes}, Mathematika \textbf{46} (1999), 103--125.

\bibitem{ludwig1999affine}
\bysame, \emph{A characterization of affine length and asymptotic approximation
  of convex discs}, Abh. Math. Semin. Univ. Hamburg \textbf{69} (1999), 75--88.

\bibitem{rockafellar2009variational}
R.~T. Rockafellar and R.~J.-B. Wets, \emph{Variational analysis}, 3rd ed.,
  Grundlehren der Mathematischen Wissenschaften, vol. 317, Springer-Verlag,
  Berlin, 2009.

\bibitem{SchuttThaeleTurchiWerner2024}
C.~Sch{\"u}tt, C.~Thaele, N.~Turchi, and E.~M. Werner, \emph{Weighted floating
  functions and weighted functional affine surface areas}, Trans. Amer. Math.
  Soc., in press.

\bibitem{trudinger2000bernstein}
N.~S. Trudinger and X.-J. Wang, \emph{The bernstein problem for affine maximal
  hypersurfaces}, Invent. Math. \textbf{140} (2000), 399--422.

\bibitem{zador1982asymptotic}
P.~L. Zador, \emph{Asymptotic quantization error of continuous signals and the
  quantization dimension}, IEEE Trans. Inform. Theory \textbf{IT-28} (1982),
  139--148.

\end{thebibliography}
\bibliographystyle{amsplain}
\end{document}